\documentclass[10pt]{amsart}

\usepackage{amsmath}
  \usepackage{paralist}
  \usepackage{graphics} 
  \usepackage{epsfig} 
\usepackage{graphicx}  \usepackage{epstopdf}

 \usepackage[colorlinks=true]{hyperref}
\hypersetup{urlcolor=blue, citecolor=red}

\newtheorem{theorem}{Theorem}[section]
\newtheorem{corollary}[theorem]{Corollary}
\newtheorem{main}{Main Theorem}
\newtheorem*{main*}{Main Theorem}
\newtheorem{lemma}[theorem]{Lemma}
\newtheorem{proposition}[theorem]{Proposition}

\theoremstyle{definition}
\newtheorem{definition}[theorem]{Definition}
\newtheorem{remark}[theorem]{Remark}
\newtheorem{example}[theorem]{Example}

\title[Running heading with forty characters or less]
      {Modified Schmidt games and non-dense forward orbits of  partially hyperbolic systems}

\author[first-name1 last-name1 and first-name2 last-name2]{Weisheng Wu}

\subjclass{}
 \keywords{Modified Schmidt games; Hausdorff dimension; Partially hyperbolic diffeomorphisms; Non-dense orbits}

\address{School of Mathematical Sciences, Peking University, Beijing, 100871, China}
 \email{wuweisheng@math.pku.edu.cn}

\begin{document}
\maketitle
\markboth{Modified Schmidt games and non-dense forward orbits of  partially hyperbolic systems}
{Modified Schmidt games and non-dense forward orbits of  partially hyperbolic systems}
\renewcommand{\sectionmark}[1]{}

\begin{abstract}
Let $f: M \to M$ be a $C^{1+\theta}$-partially hyperbolic diffeomorphism. We introduce a type of modified Schmidt games which is induced by $f$ and played on any unstable manifold. Utilizing it we generalize some results of \cite{Wu} as follows. Consider a set of points with non-dense forward orbit:
$$E(f, y) := \{ z\in M: y\notin \overline{\{f^k(z), k \in \mathbb{N}\}}\}$$
for some $y \in M$ and
$$E_{x}(f, y) := E(f, y) \cap W^u(x)$$
for any $x\in M$. We show that $E_x(f,y)$ is a winning set for such modified Schmidt games played on $W^u(x)$, which implies that $E_x(f,y)$ has Hausdorff dimension equal to $\dim W^u(x)$. Then for any nonempty open set $V \subset M$ we show that $E(f, y) \cap V$ has full Hausdorff dimension equal to $\dim M$, by using a technique of constructing measures supported on $E(f, y)$ with lower pointwise dimension approximating $\dim M$.
\end{abstract}

\section{Introduction}
\subsection{Non-dense orbits and Schmidt games}
It is well known that an Anosov system has many periodic points and many other proper compact invariant sets (cf. \cite{Pr}, \cite{Fr}, \cite{Ma}), which demonstrate high complexity of such systems. For partially hyperbolic systems, we even don't know the existence of periodic points. Nevertheless, we can study non-dense orbits of a partially hyperbolic system in order to study its proper compact invariant sets. In this paper, we show that for a general partially hyperbolic system the set of points with non-dense forward orbit has full Hausdorff dimension equal to the dimension of the underlying manifold. This is an interesting result, because it is obvious to see that the set of points with non-dense forward orbit is negligible with respect to any smooth ergodic measure (if we suppose such a measure exists). But this exceptional set is large in terms of Hausdorff dimension.

The same result was proved in \cite{U} for a $C^2$-expanding endomorphism or a transitive $C^2$-Anosov diffeomorphism by using Markov partitions. Another remarkable result is that for any toral endomorphism on $\mathbb{T}^n$ the set of points with non-dense forward orbit is a winning set for Schmidt games (cf. \cite{D}, \cite{BFK}). The winning property is a strengthening of having full Hausdorff dimension, for example it is stable under countable intersections. It was also proved in \cite{tseng} that for a $C^2$-expanding self-maps on the circle the set of points with non-dense forward orbit is a winning set for Schmidt games.

Schmidt games were first introduced by W.M.Schmidt in \cite{S} in 1966. A winning set for such games is large in the following sense: it is dense in the metric space, and its intersection with any nonempty open subset has full Hausdorff dimension. Moreover, the winning property is stable with respect to countable intersections. See Section $2$ or \cite{S} for more details. Schmidt also proved in \cite{S} that the set of badly approximable numbers is a winning set for Schmidt games and hence has full Hausdorff dimension $1$. Schmidt games find their novelty and importance in homogeneous dynamics due to the well known connection between Diophantine approximation and bounded orbits of nonquasiunipotent flows on homogeneous spaces. For topics along this line, see \cite{S2}, \cite{D1}, \cite{D2}, \cite{AL}, \cite{Dol}, \cite{KM}, \cite{KW1}, \cite{KW2}, \cite{DS}, \cite{AN} and many others. Recently Kleinbock and Weiss developed a type of modified Schmidt games in \cite{KW1}. They showed in \cite {KW1} that the set of weighted badly approximable vectors is winning for such modified Schmidt games and hence has full Hausdorff dimension. They also proved in \cite {KW2} a conjecture of Margulis which strengthens the results in \cite{KM}, by utilizing such modified Schmidt games.

\subsection{Statement of results}
In this paper we always let $M$ be an $n$-dimensional smooth, connected and compact Riemannian manifold without boundary and let $f: M \rightarrow M$ be a $C^{1+\theta}$-diffeomorphism. $f$ is \emph{partially hyperbolic} (cf. \cite{RRU}) if there exists a nontrivial continuous $Tf$-invariant splitting of the tangent bundle $TM= E^s \oplus E^c \oplus E^u$ into so called stable, center, and unstable distributions, such that all unit vectors $v^{\sigma} \in E_x^\sigma$ ($\sigma= c,s,u$) with $x\in M$ satisfy
\begin{equation*}
\|T_xfv^s\| < \|T_xfv^c\| < \|T_xfv^u\|,
\end{equation*}
and
\begin{equation*}
\|T_xf|_{E^s_x}\| <1, \ \ \ \text{\ and\ \ \ \ } \|T_xf^{-1}|_{E^u_x}\| <1,
\end{equation*}
for some suitable Riemannian metric on $M$. The distributions $E^s$, $E^c$, $E^u$ are H\"{o}lder continuous over $M$ but in general not Lipschitz continuous. The stable distribution $E^s$ and unstable distribution $E^u$ are integrable: there exist so called stable and unstable foliations $W^s$ and $W^u$ respectively such that $TW^s=E^s$ and $TW^u=E^u$. It is well known that the foliations $W^u$ and $W^s$ are transversally absolutely continuous with bounded Jacobians (cf. \cite{BP}, \cite{PS}, \cite{BW}). Fix $y\in M$ and define
\begin{equation*}
E(f, y):= \{ z\in M: y\notin \overline{\{f^k(z), k \in \mathbb{N}\}}\}
\end{equation*}
and
\begin{equation*}
E_{x}(f, y) := E(f, y) \cap W^u(x)
\end{equation*}
for any $x \in M$, where $x \in M$ is an arbitrary point and $W^u(x)$ is the (global) unstable manifold through $x$. By definition, any point in $E(f, y)$ has a non-dense forward orbit in $M$. In \cite{Wu} we have showed the following result:
\begin{theorem}(cf. \cite{Wu})
Let $f: M \to M$ be a $C^{1+\theta}$-partially hyperbolic diffeomorphism. Assume either
\begin{enumerate}
  \item $f$ has one dimensional unstable distribution $E^u$;
OR
  \item if $\dim E^u \geq 2$, $f$ is conformal on unstable manifolds, i.e., for each $x\in M$, the derivative map $T_xf|_{E_x^u}$ is a scalar multiple of an isometry.
\end{enumerate}
Then
\begin{enumerate}
  \item $E_{x}(f, y)$ is a winning set for Schmidt games played on $W^u(x)$;
  \item for any nonempty open subset $V$ of $M$, $\dim_H(E(f, y)\cap V)=n$, where $n=\dim M$.
\end{enumerate}
\end{theorem}

However, for general partially hyperbolic systems without conformality on unstable manifolds, we are not able to prove the winning property of $E_x(f,y)$ due to the difficulty to compare diameters of shrinking balls in Schmidt games and preimages (under $f$) of an open set that we try to avoid. Inspired by \cite{KW1}, in this paper we introduce a type of modified Schmidt games played on $W^u(x)$. The idea is that we replace shrinking balls in original Schmidt games by preimages of atoms in a tiling of $W^u(x)$. In \cite{KW1}, similar modified Schmidt games are constructed on the unstable horospherical subgroup $H$ of a nonquasiunipotent flow in a homogeneous space. The construction there relies on the fact there is a tessellation on $H$ (cf. Proposition 3.3 in \cite{KM}). Without the nice algebraic structure, our constructions will lie on all images of $W^u(x)$, i.e. $W^u(x), W^u(f(x)), W^u(f^2(x)), \cdots$. On each of these unstable manifolds we construct a (local) tiling with atoms of diameters bounded from above and below. We do not require these atoms to have Markov property under $f$. Nevertheless we are still able to have a quantitative estimate for sizes of preimages of these atoms under $f$. See Section 2 for details of the construction of such modified Schimidt games.

Utilizing such modified Schmidt games, we can prove

\begin{main}\label{main}
Let $f: M \to M$ be a $C^{1+\theta}$-partially hyperbolic diffeomorphism. Then $E_{x}(f, y)$ is a winning set for modified Schmidt games played on $W^u(x)$.
\end{main}

As a consequence of Main Theorem \ref{main} and the properties of a winning set for modified Schmidt games, one has:
\begin{corollary}\label{corollary}
$E_{x}(f, y)$ is dense in $W^u(x)$ and for any nonempty open subset $U$ of $W^u(x)$, $\dim_H(E_{x}(f, y)\cap U)=\dim W^u(x)$.
\end{corollary}

Let $Y$ be a countable subset of $M$ and define similarly:
\begin{equation*}
\begin{aligned}
E(f, Y)&:= \{ z\in M: Y \cap \overline{\{f^k(z), k \in \mathbb{N}\}}= \emptyset\}, \\
E_{x}(f, Y) &:= E(f, Y) \cap W^u(x).
\end{aligned}
\end{equation*}
By the fact that the winning property for modified Schmidt games is stable under countable intersections, we can prove:
\begin{corollary}\label{corollary2}
$E_{x}(f, Y)$ is a winning set for modified Schmidt games played on $W^u(x)$. Consequently, it is dense in $W^u(x)$ and for any nonempty open subset $U$ of $W^u(x)$, $\dim_H(E_{x}(f, Y)\cap U)=\dim W^u(x)$.
\end{corollary}

We also show that $E(f, y) \cap V$ has full Hausdorff dimension equal to $\dim M$ for any nonempty open set $V \subset M$. As in \cite{Wu}, we use a technique of constructing measures supported on $E(f, y)\cap V$ with lower pointwise dimension approximating $\dim M$. To do it, we construct a family of measures supported on $E_x(f,y)$ which has appeared in the proof of a result of McMullen (cf. \cite{Mc}, \cite{U}). More importantly we can choose these measures so that they vary measurably as $x$ varies with respect to the volume measure of a local manifold transverse to unstable manifolds. To accomplish it we must set up our modified Schmidt games on different local unstable manifolds with atoms varying at least measurably. As a result, we can prove:

\begin{main}\label{main2}
Let $f: M \to M$ be a $C^{1+\theta}$-partially hyperbolic diffeomorphism. Then for any nonempty open subset $V$ of $M$,
$$\dim_H(E(f, y)\cap V)=n$$
where $n=\dim M$.
\end{main}

\begin{corollary} \label{countable}
For any nonempty open subset $V$ of $M$, $\dim_H(E(f, Y)\cap V)=n$.
\end{corollary}

To keep consistency, we will follow the proof scheme and notations in \cite{Wu}. We describe the setup for modified Schmidt games and the properties of a winning set in Section $2$. In Section $3$, we present a winning strategy for modified Schmidt games played on $W^u(x)$ to prove Main Theorem \ref{main}. Corollaries \ref{corollary} and \ref{corollary2} are also proved there. In Section $4$, we prove Main Theorem \ref{main2}. Corollary \ref{countable} is proved at the end.

We write $\sigma=\dim E^\sigma$, where $\sigma=u, c, s$. If $W$ is a submanifold of $M$, then $B^W(x,r)$ denotes the open ball centered at $x$ and of radius $r$ in $W$, with respect to the induced Riemannian metric on $W$. Similarly $B^u(x,r)$ denotes an open ball in $W^u(x)$. We always denote as $\nu$ the volume measures on various manifolds if it doesn't cause confusion. We suppose that $1< \sigma_1\leq \|(Tf|W^u)(z)\| \leq \sigma_2$ for any $z \in M$.

\section{Modified Schmidt games}
\subsection{Schmidt Games}

Let $(X, d)$ be a complete metric space. We denote as $B(x,r)$ the ball of radius $r$ with center $x$. If $\omega=(x,r) \in X \times \mathbf{R}_{+}$, we also write $B(\omega):=B(x,r)$.

Schmidt games are played by two players, Alice and Bob. Fix $0 < \alpha, \beta <1$ and a subset $S \subset X$ (the target set). Bob starts the game by choosing $x_1 \in X$ and $r_1 >0$ hence specifying a pair $\omega_1=(x_1, r_1)$. Then Alice chooses a pair $\omega'_1=(x_1', r_1')$ such that $B(\omega'_1) \subset B(\omega_1)$ and $r_1'=\alpha r_1$. In the second turn, Bob chooses a pair $\omega_2=(x_2, r_2)$ such that $B(\omega_2) \subset B(\omega_1')$ and $r_2= \beta r_1'$, and so on. In the $k$th turn, Bob and Alice choose $\omega_k=(x_k, r_k)$ and $\omega'_k=(x_k', r_k')$ respectively such that
\begin{equation*}
  B(\omega'_k) \subset B(\omega_k) \subset B(\omega_{k-1}'), \ \ r_k= \beta r_{k-1}', \ \ r_k'=\alpha r_k.
\end{equation*}
Thus we have a nested sequence of balls in $X$:
\begin{equation*}
B(\omega_1) \supset B(\omega_1') \supset \cdots \supset B(\omega_k) \supset B(\omega_k') \supset \cdots.
\end{equation*}
The intersection of all these balls consists of a unique point $x_{\infty} \in X$. We call Alice the winner if $x_{\infty} \in S$, and Bob the winner otherwise. $S$ is called an $(\alpha, \beta)$-winning set if Alice has a strategy to win regardless of how well Bob plays, and we call such a strategy an $(\alpha, \beta; S)$-winning strategy. $S$ is called $\alpha$-winning if it is $(\alpha, \beta)$-winning for any $0 <\beta <1$. $S$ is called a winning set if it is $\alpha$-winning for some $0 < \alpha <1$.

The following nice properties of a winning set are proved in \cite{S}.

\begin{proposition}
Some properties of winning sets for Schmidt games:
\begin{enumerate}
\item If the game is played on $X=\mathbb{R}^n$ with the Euclidean metric, then any winning set is dense and has full Hausdorff dimension $n$.
\item The intersection of countably many $\alpha$-winning sets is $\alpha$-winning.
\end{enumerate}
\end{proposition}

\subsection{Modified Schmidt games}
\subsubsection{A family of tilings}
In this subsection we describe the setups for modified Schmidt games. We postpone the definition of modified Schmidt games to the next subsection.
\begin{definition}\label{tiling}
Let $\delta, \tau >0$ be very small and let $D_0$ be an open and connected subset in $W^u(x)$ satisfying $B^u(x_0, \frac{(1-\tau)\delta}{2}) \subset D_0 \subset B^u(x_0, (1+\tau)\delta)$ for some $x_0\in W^u(x)$. If there exists a family of subsets $D_n^i$ ($n=1,2,\cdots $, $i=1,2,\cdots, k_n$) of $W^u(x)$ such that
 \begin{enumerate}
\item each $D_n^i$ is an open and connected subset of $W^u(x)$ satisfying
  $$B^u(f^n(x_n^i), \frac{(1-\tau)\delta}{2}) \subset f^n(D_n^i) \subset B^u(f^n(x_n^i), (1+\tau)\delta)$$
  for some $x_n^i \in W^u(x)$,
\item $D_n^i \cap D_n^j=\emptyset$ for every $n$ and $i \neq j$,
\item $D_0 \subset \bigcup_{i=1}^{k_n} \overline{D_n^i}$ and $D_n^i \cap D_0 \neq \emptyset$ for every $n$,
\end{enumerate}
then we say that $\{D_n^i\}$ form a family of $f$-induced $(D_0, \delta, \tau)$-local tilings on $W^u(x)$. We denote it as $\mathcal{T}$ which consists of countable local tilings $\{\mathcal{T}_n\}_{n=0}^{\infty}$ with $\mathcal{T}_n:=\{D_n^i\}_{i=1}^{k_n}$. Denote $\mathbf{T}_n=\bigcup_{i=1}^{k_n}D_n^i$ and $\mathbf{T}=\bigcup_{n=0}^\infty \mathbf{T}_n$. We call each $D_n^i$ an $n$-atom of $\mathcal{T}$.
\end{definition}

\begin{example}\label{exampletiling}
Let $S_n(\delta)$ be a maximal $\delta$-separated set in $f^n(D_0)$. For each $z\in S_n(\delta)$, define
$$D_n(z):=\{y\in W^u(f^n(x)): d(z,y)< d(w,y) \text{\ for any\ } w \in S_n(\delta) \text{\ and\ } w\neq z\}.$$
Then $\{f^{-n}(D_n(z)):z\in S_n(\delta)\}$ form a family of $f$-induced $(D_0, \delta, \tau)$-local tilings on $W^u(x)$.
\end{example}
\begin{proof}
(2) and (3) in Definition \ref{tiling} is obvious. For (1), it suffices to show that $B^u(z,\frac{\delta}{2}) \subset D_n(z) \subset B^u(z,\delta)$. Let $y\in B^u(z,\frac{\delta}{2})$, i.e. $d(z,y) < \frac{\delta}{2}$. Assume that $d(w,y)<\frac{\delta}{2}$ for some $w \in S_n(\delta)$ and $w\neq z$. One has that $d(z,w)\leq d(z,y)+d(w,y)<\frac{\delta}{2}+\frac{\delta}{2}=\delta$ which is impossible. Thus $d(z,y)<\frac{\delta}{2}\leq d(w,y)$ for any $w \in S_n(\delta)$ and $w\neq z$. Therefore, $B^u(z,\frac{\delta}{2}) \subset D_n(z)$.

On the other hand, let $y\in D_n(z)$. Assume that $d(z,y) \geq \delta$. Then $d(w,y) > d(z,y) \geq \delta$ for any $w \in S_n(\delta)$ and $w\neq z$. This is impossible since $S_n(\delta)$ is a maximal $\delta$-separated set. Therefore, $D_n(z) \subset B^u(z,\delta)$.
\end{proof}

\begin{remark}
In Example \ref{exampletiling}, we can choose $\tau >0$ to be arbitrarily small. The utility of $\tau$ will appear in Section 4 to guarantee the stability of a family of $f$-induced local tilings on $W^u(x)$ under perturbation.
\end{remark}
\begin{remark}
We have much freedom to choose a family of $f$-induced local tilings on $W^u(x)$. We do not require Markov property under $f$ for it. The essential property we will use is tiling and that the image of each atom under some iterate of $f$ has diameter comparable to $\delta$.
\end{remark}

Now let us fix a family of $f$-induced $(D_0, \delta, \tau)$-local tilings $\mathcal{T}$ on $W^u(x)$. If $z\in \mathbf{T}_n$ we use $D_n^i(z)$ to denote the atom in $\mathcal{T}$ which contains $z$. We also write $\psi(z,n)=D_n^i(z)$. We can define a partial order on $\Omega=\{(z,n):z\in \mathbf{T}_n\}$:
$$(z,n)\leq (z',n') \Leftrightarrow \psi(z,n) \subset \psi(z',n').$$
We study the following crucial properties of $\mathcal{T}$ before defining modified Schmidt games. We will follow the terminology in \cite{KW1}.
\begin{flushleft}
\textbf{(MSG0)} There exists $a_* \in \mathbb{N}_+$ such that the following property holds: for any $(z,n)\in \Omega$ and any $m>a_*$ there exists $z'\in \mathbf{T}_{m+n}$ such that $(z',m+n) \leq (z,n)$.
\end{flushleft}
\begin{flushleft}
  \textbf{(MSG1)} For any nonempty open set $U \subset D_0$, there is $(z,n)$ such that $\psi(z,n) \subset U$.
\end{flushleft}
\begin{flushleft}
\textbf{(MSG2)} There exist $C, \sigma >0$ such that $\text{diam}(\psi(\omega)) \leq C\exp(-\sigma n)$ for all $n > 0,  \omega=(z,n)$ where $z\in \mathbf{T}_n$.
\end{flushleft}
Let $\nu$ be the volume measure on $W^u(x)$. In addition, we state the following two properties:
\begin{flushleft}
\textbf{($\nu 1$)} $\nu(\psi(z,n)) >0, \text{\ for any\ } (z,n)\in \Omega$.
\end{flushleft}
We always choose $\delta>0$ in Definition \ref{tiling} to be small enough.
   \begin{flushleft}
    \textbf{($\nu 2$)} For any $a > a_*$, there exists $c=c(a)> 0$ satisfying the following property: for any $\omega=(z,n)\in \Omega$ and any $b>a_*$, there exist $\theta_i=(z_i,n+b)\in \Omega, \ i= 1, \cdots, N$ such that
    \begin{enumerate}
  \item $\psi(\theta_1), \cdots, \psi(\theta_N) \subset \psi(\omega)$ and they are essentially disjoint;
  \item for every $\psi(\theta_i') \subset \psi(\theta_i)$ where $\theta_i'=(z_i',n+a+b)$ one has
  $$\nu(\bigcup_{i=1}^N\psi(\theta_i')) \geq c\nu(\psi(\omega)).$$
\end{enumerate}
   \end{flushleft}

\begin{proposition}\label{gameproperty}
$\mathcal{T}$ satisfies properties (MSG 0-2) and ($\nu$ 1-2).
\end{proposition}
Before proving Proposition \ref{gameproperty}, we state the following bounded distortion lemma:
\begin{lemma}\label{BD}(Bounded distortion)
Let $c>0$ be small enough and $k \geq 0$. For any $z_1, z_2$ with $f^k(z_1)\in B^u(f^k(z_2),c)$, one has
$$\frac{1}{K} \leq \frac{\text{Jac}(f^k|W^u)(z_1)}{\text{Jac}(f^k|W^u)(z_2)} \leq K$$
for some $K=K(c)$, and $K \to 1$ as $c \to 0$.
\end{lemma}

\begin{proof}
Since $z\mapsto \text{Jac}(f|W^u)(z)$ is H\"{o}lder continuous, $\log \text{Jac}(f|W^u)$ is as well and thus there exist $l>0$, $0 < \theta' <1$ such that
$$\left|\log \text{Jac}(f|W^u)(z_1)-\log \text{Jac}(f|W^u)(z_2)\right| \leq l (d^u(z_1, z_2))^{\theta'}$$
for nearby $z_1$ and $z_2$. For any $z_1, z_2$ with $f^k(z_1)\in B^u(f^k(z_2),c)$, one has
\begin{equation*}
d^u(f^i(z_1), f^i(z_2)) \leq \frac{c}{(\sigma_1)^{k-i}}, \ \ \text{for\ \ } \forall 0 \leq i \leq k.
\end{equation*}
Thus,
\begin{equation*}
\begin{aligned}
\left|\log \frac{\text{Jac}(f^k|W^u)(z_1)}{\text{Jac}(f^k|W^u)(z_2)}\right| &\leq \sum_{i=0}^{k-1}\left|\log \text{Jac}(f|W^u)(f^i(z_1)) - \log \text{Jac}(f|W^u)(f^i(z_2))\right|
 \\ &\leq \sum_{i=0}^{k-1}l(d^u(f^i(z_1), f^i(z_2)))^{\theta'}
\\&\leq \sum_{i=0}^{k-1}\frac{lc^{\theta'}}{(\sigma_1)^{\theta' (k-i)}} \leq \frac{lc^{\theta'}}{(\sigma_1)^{\theta'}-1}.
\end{aligned}
\end{equation*}
Hence $$\frac{1}{K} \leq \frac{\text{Jac}(f|W^u)(z_1)}{\text{Jac}(f|W^u)(z_2)}\leq K,$$
where $K=\exp(\frac{lc^{\theta'}}{(\sigma_1)^{\theta'}-1})$.
\end{proof}

\begin{proof}[Proof of Proposition \ref{gameproperty}]
(MSG0): We take $a_*$ large enough such that
\begin{equation}\label{e:astar}
\frac{(1+\tau)\delta}{\sigma_1^{a_*}} < \frac{(1-\tau)\delta}{4}.
\end{equation}
Let $m>a_*$. Consider the atom $f^n(\psi(z,n))$ on $W^u(f^n(x))$. By definition there exists $z_n^i\in W^u(x)$ such that
$$B^u(f^n(z_n^i), \frac{(1-\tau)\delta}{2}) \subset f^n(\psi(z,n)) \subset B^u(f^n(z_n^i), (1+\tau)\delta).$$
So there exists $w\in f^n(\psi(z,n))$ such that $d^u(w, \partial f^n(\psi(z,n)))>\frac{(1-\tau)\delta}{2}$ and $f^m(w) \in f^{m+n}(D_{m+n}^j)$ for some $D_{m+n}^j\in \mathcal{T}_{m+n}$. So $w\in f^n(D_{m+n}^j)$ and by \eqref{e:astar}
$$\text{diam}(f^n(D_{m+n}^j)) \leq\frac{2(1+\tau)\delta}{\sigma_1^m} <\frac{2(1+\tau)\delta}{\sigma_1^{a_*}}< \frac{(1-\tau)\delta}{2}.$$
This implies that $f^n(D_{m+n}^j) \subset f^n(\psi(z,n))$. Thus $D_{m+n}^j \subset \psi(z,n)$.

(MSG1): This is clear by the tiling property (3) in Definition \ref{tiling} and the fact that the diameter of an $n$-atom decreases to zero as $n\to \infty$.

(MSG2): By property (1) in Definition \ref{tiling} and $\|(Tf|W^u)(z)\| \geq \sigma_1$ for any $z\in M$, one has
$$\text{diam}(\psi(z,n)) \leq\frac{2(1+\tau)\delta}{\sigma_1^n}.$$
Set $C=2(1+\tau)\delta$ and $\sigma=\log \sigma_1$, then we have $\text{diam}(\psi(z,n))\leq C\exp(-\sigma n)$.

($\nu$1): It is clear beacuse $\psi(z,n)$ is a nonempty open subset of $W^u(x)$.

($\nu$2): Let us choose $\delta, \tau>0$ in Definition \ref{tiling} to be small enough. We take $\theta_1=(z_1,n+b), \cdots, \theta_N=(z_N,n+b)\in \Omega$ to be all the $(n+b)$-atoms with $\psi(\theta_i) \subset \psi(\omega)$. Recall that $B^u(f^n(x_n^i), \frac{(1-\tau)\delta}{2}) \subset f^n(\psi(\omega)) \subset B^u(f^n(x_n^i), (1+\tau)\delta)$ for some $x_n^i \in W^u(x)$. Note that for any $\psi(\theta) \in \mathcal{T}_{n+b}$ with $f^n(\psi(\theta))\cap \partial (f^n(\psi(\omega)))\neq \emptyset$ one has $\text{diam}(f^n(\psi(\theta)))\leq \frac{2(1+\tau)\delta}{\sigma_1^b}$. Therefore,
\begin{equation*}
\begin{aligned}
\bigcup_{i=1}^Nf^n(\psi(\theta_i)) &\supset B^u\left(f^n(x_n^i), \frac{(1-\tau)\delta}{2}-\frac{2(1+\tau)\delta}{\sigma_1^b}\right)\\
&\supset B^u\left(f^n(x_n^i), \frac{(1-\tau)\delta}{2}-\frac{2(1+\tau)\delta}{\sigma_1^{a_*}}\right).
\end{aligned}
\end{equation*}
Note that $\frac{(1-\tau)\delta}{2}-\frac{2(1+\tau)\delta}{\sigma_1^{a_*}} >0$ by \eqref{e:astar}. Then
$$\frac{\nu(\bigcup_{i=1}^Nf^n(\psi(\theta_i)))}{\nu(f^n(\psi(\omega)))}\geq C\cdot \left(\frac{\frac{(1-\tau)\delta}{2}-\frac{2(1+\tau)\delta}{\sigma_1^{a_*}}}{(1+\tau)\delta}\right)^u=C\cdot\left(\frac{(1-\tau)}{2(1+\tau)}-\frac{2}{\sigma_1^{a_*}}
\right)^u.$$
One also has
$$\frac{\nu(\bigcup_{i=1}^Nf^{n+b}(\psi(\theta'_i)))}{\nu(\bigcup_{i=1}^Nf^{n+b}(\psi(\theta_i)))}\geq C\cdot \left(\frac{\frac{(1-\tau)\delta}{2\sigma_2^{a}}}{(1+\tau)\delta}\right)^u=C\cdot\left(\frac{1-\tau}{2\sigma_2^{a}(1+\tau)}
\right)^u.$$
Finally we use the bounded distortion Lemma \ref{BD} to obtain the following estimates on volume:
\begin{equation*}
\begin{aligned}
\frac{\nu(\bigcup_{i=1}^N\psi(\theta'_i))}{\nu(\psi(\omega))}&=\frac{\nu(\bigcup_{i=1}^N\psi(\theta'_i))}{\nu(\bigcup_{i=1}^N\psi(\theta_i))}\cdot\frac{\nu(\bigcup_{i=1}^N\psi(\theta_i))}{\nu(\psi(\omega))}\\
&\geq\frac{\nu(\bigcup_{i=1}^Nf^{n+b}(\psi(\theta'_i)))}{K\nu(\bigcup_{i=1}^Nf^{n+b}(\psi(\theta_i)))}\cdot\frac{\nu(\bigcup_{i=1}^Nf^n(\psi(\theta_i)))}{K\nu(f^n(\psi(\omega)))}\\
&\geq\frac{C^2}{K^2}\cdot\left(\frac{1-\tau}{2\sigma_2^{a}(1+\tau)}
\right)^u\cdot\left(\frac{(1-\tau)}{2(1+\tau)}-\frac{2}{\sigma_1^{a_*}}\right)^u.\\
&:=c(a)
\end{aligned}
\end{equation*}
where $K=K(2(1+\tau)\delta)$ as in Lemma \ref{BD} is close to $1$ since $\delta, \tau>0$ are very small.
\end{proof}

\subsubsection{Modified Schmidt games}
We define a type of modified Schmidt games induced by $f$ on $W^u(x)$ with respect to $\mathcal{T}$ as follows. Fix $a, b \in \mathbb{N}_+$ both larger than $a_*$ and a subset $S$ of $W^u(x)$ (the target set). Bob starts the $(a,b)$-game by choosing a pair $\omega_1=(z_1,n_1)\in \Omega$, hence specifying a $\mathcal{T}_{n_1}$-atom $D_{n_1}(z_1)$. By virtue of (MSG0), Alice can choose a pair $\omega'_1=(z'_1,n'_1)$ such that $\omega'_1\leq \omega_1$ and $n'_1=n_1+a$. In the second turn, Bob chooses a pair $\omega_2=(z_2, n_2)$ such that $\omega_2 \leq \omega_1'$ and $n_2= n_1'+b$, and so on. In the $k$th turn, Bob and Alice choose $\omega_k=(z_k, n_k)$ and $\omega'_k=(z_k', n_k')$ respectively such that
\begin{equation*}
  \psi(\omega'_k) \subset \psi(\omega_k) \subset \psi(\omega_{k-1}'), \ \ n_k= n_{k-1}'+b, \ \ n_k'=n_k+a.
\end{equation*}
Note that Bob and Alice can always make their choices by virtue of (MSG0). Thus we have a nested sequence of atoms in $\mathcal{T}$:
\begin{equation} \label{e:sequence}
\psi(\omega_1) \supset \psi(\omega_1') \supset \cdots \supset \psi(\omega_k) \supset \psi(\omega_k') \supset \cdots.
\end{equation}
By (MSG2), the intersection of all these atoms consists of a unique point $z_{\infty} \in W^u(x)$ since $W^u(x)$ is a complete metric space. We call Alice the winner if $z_{\infty} \in S$, and Bob the winner otherwise. $S$ is called a $(a, b)$-winning set for \emph{modified Schmidt games} if Alice has a strategy to win regardless of how well Bob plays, and we call such a strategy an $(a, b; S)$-winning strategy. $S$ is called $a$-winning if it is $(a,b)$-winning for any $b >a_*$. $S$ is called a winning set if it is $a$-winning for some $a > a_*$.

We have the following nice properties of a winning set:
\begin{proposition}\label{winproperty}
Some properties of winning sets for modified Schmidt games induced by $f$ on $W^u(x)$ with respect to $\mathcal{T}$:
\begin{enumerate}
\item An $(a,b)$-winning set is dense in $D_0 \subset W^u(x)$.
\item The intersection of countably many $a$-winning sets is $a$-winning.
\end{enumerate}
\end{proposition}
\begin{proof}
Assume that an $(a,b)$-winning set $S$ is not dense in $D_0 \subset W^u(x)$. Then there exists a nonempty open set $U \subset D_0$ with $S\cap U =\emptyset$. By (MSG1), we can let Bob start the $(a,b)$-game by choosing an atom $\psi(z,n) \subset U$. Then Alice is impossible to win the game. This contradicts to the assumption that $S$ is $(a,b)$-winning. Thus $S$ is dense in $D_0$.

The proof of (2) is a modification of the one of Theorem 2 in \cite{S}. It will appear in the proof of Corollary \ref{countable} again, so we omit it here.
\end{proof}

Before giving the estimates on the Hausdorff dimension of a winning set, we state the following construction for producing fractal sets, which has been described in \cite{Mc}, \cite{U}, \cite{KM}, \cite{KW1}, \cite{Wu} and many others. Let $A_0 \subset X$ be a compact subset of a Riemannian manifold $X$ and let $m$ be a Borel measure on $X$. For each $l\in \mathbb{N}_0$, let $\mathcal{A}_l$ denote a finite collection of compact subsets of $A_0$ satisfying the following conditions:
\begin{equation} \label{e:tree1}
\mathcal{A}_0 =\{A_0\} \text{\ and\ } m(A_0) >0;
\end{equation}
\begin{equation} \label{e:tree2}
\text{For all\ } l \in \mathbb{N}, \ \text{if\ } A, B \in \mathcal{A}_l \text{\ and\ } A \neq B, \text{\ then\ } m(A \cap B)=0;
\end{equation}
\begin{equation} \label{e:tree3}
\text{For all\ } l \in \mathbb{N}, \text{\ every element }B \in \mathcal{A}_l \text{\ is contained in an element\ } A \in \mathcal{A}_{l-1}.
\end{equation}
Let $\mathcal{A}$ be the union of subcollections $\mathcal{A}_l, l\in \mathbb{N}_0$. Then $\mathcal{A}$ is called \emph{tree-like} if it satisfies conditions \eqref{e:tree1},\eqref{e:tree2},\eqref{e:tree3}. $\mathcal{A}$ is called \emph{strongly tree-like} if it is tree-like and in addition:
\begin{equation} \label{e:tree4}
d_l(\mathcal{A}):= \sup\{\text{diam}(A): A \in \mathcal{A}_l\} \to 0 \text{\ as\ } l \to \infty.
\end{equation}

For each $l \in \mathbb{N}_0$, denote $\mathbf{A}_l=\cup_{A\in \mathcal{A}_l}A$. Then we can define the limit set of $\mathcal{A}$ to be
$$\mathbf{A}_{\infty}:=\bigcap_{l\in \mathbb{N}_0} \mathbf{A}_l.$$
For any Borel $B \subset A_0$ with $m(B) >0$ and $l \in \mathbb{N}$, define the $l$th stage density $\delta_l(B, \mathcal{A})$ of $\mathcal{A}$ in $B$ by
\begin{equation*}
\delta_l(B, \mathcal{A})=\frac{m(\mathbf{A}_l \cap B)}{m(B)}.
\end{equation*}
Assume that
\begin{equation} \label{e:tree5}
\Delta_l(\mathcal{A}):=\inf_{B\in \mathcal{A}_l}\delta_{l+1}(B, \mathcal{A}) >0.
\end{equation}

We state the following lemma without proof. Note that it is also essential to our construction of measures in Section 4. For its proof, see \cite{Mc}, \cite{U}, \cite{Wu}, etc.
\begin{lemma} \label{treelike}
Let $\mathcal{A}$ be defined as above, satisfying conditions \eqref{e:tree1}--\eqref{e:tree5}. Assume that there exist constants $D>0$ and $k>0$ such that for any $z \in A_0$,
\begin{equation*}
m(B(z,r)) \leq Dr^k.
\end{equation*}
Denote
\begin{equation*}
\epsilon:=\limsup_{l \to \infty}\frac{\sum_{i=0}^{l}\log(\frac{1}{\Delta_i(\mathcal{A})})}{\log(\frac{1}{d_l(\mathcal{A})})},
\end{equation*}
then there exists a sequence of measures $\mu^{(l)}$ supported on $\mathbf{A}_l$ such that:
\begin{enumerate}
  \item The sequence $\mu^{(l)}$ has a unique limit measure $\bar{\mu}$, which is supported on $\mathbf{A}_{\infty}$;
  \item $\bar{\mu}(B(z,r)) \leq Cr^{k-\epsilon}$ for any $z\in \mathbf{A}_\infty$, $r>0$, and some $C>0$;
  \item  $\dim_H(\mathbf{A}_{\infty}) \geq k-\epsilon$.
\end{enumerate}

\end{lemma}

\begin{proposition} \label{schmidtgame}
Let $S\subset W^u(x)$ be an $a$-winning set for modified Schmidt games induced by $f$ with respect to $\mathcal{T}$. Then
\begin{enumerate}
  \item for any open $U\subset D_0$, one has
\begin{equation*}
\dim_H(S \cap U) = u;
\end{equation*}
  \item for any $\epsilon>0$, there exist a sequence of measures $\{\mu_x^{(l)}\}_{l=0}^{\infty}$ and $\{\mu_x\}$ such that:
\begin{enumerate}
  \item $\mu_x^{(l)}\text{\ is supported on\ } D_0 \text{\ and\ }\mu_x \text{\ is supported on\ } S \subset D_0$,
  \item $\mu_x \text{\ is the unique weak limit of\ } \mu_x^{(l)}$,
  \item for any $z\in D_0\cap \text{supp}(\mu_x)$, $r>0$ small enough,
\begin{equation*}
\mu_x (B^u(z, r)) \leq Cr^{u-\epsilon}.
\end{equation*}
\end{enumerate}
\end{enumerate}

\end{proposition}

\begin{proof}
Fix $a,b>a_*$ and consider $(a, b)$-modified Schmidt games. We will construct a family $\mathcal{A}$ satisfying \eqref{e:tree1}--\eqref{e:tree5} whose limit set $\mathbf{A}_{\infty}$ is a subset of $S\cap U$. It will be constructed by considering possible moves for Bob at each stage and Alice's corresponding counter-moves, which give us different sequences as in \eqref{e:sequence}. We will take the measure $m$ in Lemma \ref{treelike} to be the volume measure $\nu$ on $W^u(x)$. By (MSG1), Bob may begin the game by choosing an atom $\psi(\omega_1) \subset U$. Since $S$ is winning, Alice can choose a ball $\psi(\omega_1')$ which has nonempty intersection with $S$. Take $A_0:= \psi(\omega_1')$; then \eqref{e:tree1} is satisfied by \textbf{($\nu 1$)}.

By ($\nu$2), $B(\omega_1')$ contains $N=N(\omega_1')$ essentially disjoint atoms $\psi(\omega_2^{(1)}),\cdots, \psi(\omega_2^{(N)})$, and each of them could be chosen by Bob as $\psi(\omega_2)$. For each of such choice $\psi(\omega_2^{(i)})$ of Bob, Alice can pick an atom $\psi((\omega')_2^{(i)}) \subset \psi(\omega_2^{(i)})$. Let $\mathcal{A}_1$ be the collection of $N(\omega_1')$ atoms $\psi((\omega')_2^{(i)})$, i.e., the atoms chosen by Alice. Repeating the same for each turn of the game, we obtain $\mathcal{A}_2$, $\mathcal{A}_3$, etc. \eqref{e:tree2} and \eqref{e:tree3} are immediate from the construction above. By (MSG2), $d_l(\mathcal{A}) \leq C\exp(-\sigma(n_1+a+l(a+b)))\to 0$ as $l\to \infty$. Thus we get \eqref{e:tree4}. By ($\nu$2) one has
\begin{equation*}
\Delta_l(\mathcal{A})=\inf\frac{\nu\left(\bigcup_{i=1}^{N(\omega'_{l+1})}\psi(\omega'_{l+2})\right)}{\nu(\psi(\omega'_{l+1}))}\geq c(a)>0.
\end{equation*}
Thus we get \eqref{e:tree5}. By Lemma \ref{treelike},
\begin{equation*}
\begin{aligned}
\dim_H(\mathbf{A}_{\infty})
&\geq u-\limsup_{l \to \infty}\frac{\sum_{i=0}^{l}\log(\frac{1}{\Delta_i(\mathcal{A})})}{\log(\frac{1}{d_l(\mathcal{A})})}
\\ &\geq u-\limsup_{l \to \infty}\frac{(l+1)\log \frac{1}{c(a)}}{-\log C+\sigma(n_1+a+l(a+b))}
\\&=u-\frac{\log \frac{1}{c(a)}}{\sigma(a+b)}
\\ & \to u  \ \ \ \text{as\ } b \to +\infty.
\end{aligned}
\end{equation*}
For any $\epsilon >0$, we choose $b$ large enough such that $\frac{\log \frac{1}{c(a)}}{\sigma(a+b)}<\epsilon$. Application of Lemma \ref{treelike} also gives us the measures $\{\mu_x^{(l)}\}_{l=0}^{\infty}$ and $\{\mu_x\}$ as required in (2) of the present proposition.
\end{proof}

\section{$E_x(f,y)$ is winning}
In this section, we prove that the set $E_x(f,y)$ is a winning set for modified Schmidt games induced by $f$ on $W^u(x)$ with respect to $\mathcal{T}$, which is defined in the previous section. The following lemma is useful.

\begin{lemma}\label{avoid}
Let $a\in \mathbb{N}_+$ be such that
\begin{equation}\label{e:a}
a > \frac{\log(\frac{13+11\tau}{1-\tau})}{\log\sigma_1}.
\end{equation}
Then there exists $0<\eta<\frac{1}{4}$ such that if $(z,n) \in \Omega$, and $B_1, \cdots, B_N$ are $N$ subsets in $D_0$ with $\text{diam}(B_i) < \frac{(1-\tau)\delta}{2\sigma_2^{n+a}}$, there exists $(z', n+a)\in \Omega$ such that
$$\psi(z',n+a) \subset \psi(z,n),$$
and
$$\psi(z',n+a)\cap B_i=\emptyset$$
for at least $\lceil \eta N \rceil$ (the smallest integer $\geq \eta N$) of the sets $B_i, \ 1\leq i \leq N$. \end{lemma}
\begin{proof}
Consider the atom $f^n(\psi(z,n)) \subset W^u(f^n(x))$. By Definition \ref{tiling}
$$B^u(f^n(x_n^i), \frac{(1-\tau)\delta}{2}) \subset f^n(D_n^i) \subset B^u(f^n(x_n^i), (1+\tau)\delta)$$
for some $x_n^i\in W^u(x)$. Since any $f^n(\psi(z',n+a))$ intersecting with $\partial(f^n(\psi(z,n)))$ has diameter no more than $\frac{2(1+\tau)\delta}{\sigma_1^a}$, one can find a point $z_n^i \in T_{n+a}$ such that $\psi(z_n^i,n+a) \subset \psi(z,n)$ and
$$d(f^n(z_n^i), f^n(x_n^i)) \geq\frac{(1-\tau)\delta}{2}-\frac{2(1+\tau)\delta}{\sigma_1^a}.$$
Therefore,
\begin{equation}\label{e:distance}
d(f^n(\psi(z_n^i,n+a), f^n(D_{n+a}^i)) \geq \frac{(1-\tau)\delta}{2}-3\cdot\frac{2(1+\tau)\delta}{\sigma_1^a}
\end{equation}
where $x_n^i \in \overline{D_{n+a}^i}$. There are two different cases:
\begin{enumerate}
  \item $\psi(z_n^i,n+a)\cap B_i=\emptyset$ for at least $\lceil \eta N \rceil$ of $B_i, \ 1\leq i \leq N$. Then we are done by setting $z'=z_n^i$.
  \item $\psi(z_n^i,n+a)\cap B_i\neq \emptyset$ for at least $N-\lceil \eta N \rceil+1$ of $B_i, \ 1\leq i \leq N$. We claim that each of these $(N-\lceil \eta N \rceil+1)$ many $B_i$'s satisfies $B_i\cap D_{n+a}^i=\emptyset$. Indeed, any $f^n(B_i)$ can not intersect both $f^n(\psi(z_n^i,n+a)$ and $f^n(D_{n+a}^i)$ because
      \begin{equation*}
      \begin{aligned}
      \text{diam}(f^n(B_i)) &\leq \frac{(1-\tau)\delta}{2\sigma_2^{n+a}}\cdot \sigma_2^n \leq  \frac{(1-\tau)\delta}{2\sigma_1^{a}}\\
      &<\frac{(1-\tau)\delta}{2}-3\cdot\frac{2(1+\tau)\delta}{\sigma_1^a} \\
      &\leq d(f^n(\psi(z_n^i,n+a), f^n(D_{n+a}^i))
      \end{aligned}
      \end{equation*}
      by combining \eqref{e:a} and \eqref{e:distance}. Let us verify that
      \begin{equation}\label{e:count}
      N-\lceil \eta N \rceil +1 \geq \lceil \eta N \rceil.
       \end{equation}
    If $N\geq 2$, \eqref{e:count} holds when $0<\eta<\frac{1}{2}-\frac{1}{2N}$. If $N=1$, it holds for every $0<\eta<1$. So if $0<\eta<\frac{1}{4}$, \eqref{e:count} always holds. So we are done by setting $\psi(z',n+a)=D_{n+a}^i$.
\end{enumerate}
\end{proof}
\begin{remark}
One may have an upper bound for $\eta$ larger than $\frac{1}{4}$ from the above proof.
\end{remark}
\begin{remark}
By \eqref{e:a}, we must choose different values for $a$ under different partially hyperbolic diffeomorphisms.
\end{remark}

Let $W$ be a local manifold passing through $y$ transversally to the foliation $W^u$ ($\dim W=n-u$). We call the set
\begin{equation*}
\Pi(c) := \Pi(y, W, c):= \bigcup_{z\in B^W(y,c/2)} B^u(z, c/2)
\end{equation*}
an open $c$-rectangle at $y$ ($c$ is very small). Let $I_k=I_k(c)$ denote a connected component of $f^{-k}(\Pi(c)) \cap W^u(x)$ on $W^u(x)$, $k \geq 0$. Note that there may exist more than one connected component for a single $k$. We state Main Theorem \ref{main} more precisely as follows.
\begin{theorem}\label{winning}
Let $a$ be as in \eqref{e:a}. Then $E_x(f,y)$ is $a$-winning for modified Schmidt games induced by $f$ on $W^u(x)$ with respect to $\mathcal{T}$.
\end{theorem}
\begin{proof}
Let $a$ be as in \eqref{e:a} and let $b>a_*$ be arbitrary. Let $r\in \mathbb{N}_+$ be large enough with
$$(1-\eta)^r(a+b)r<1$$
where $\eta$ is from Lemma \ref{avoid}. Fix $L>0$ to be very small. Regardless of the initial moves of Bob, Alice can make arbitrary moves waiting until Bob chooses an atom $\psi(z_1, n_1)$ with $n_1$ large enough such that
\begin{equation}\label{e:n1}
\frac{2(1+\tau)\delta}{\sigma_1^{n_1-(a+b)r}}\leq \frac{L}{100}.
\end{equation}
Let $c>0$ be small enough such that
\begin{enumerate}
  \item $c \leq \frac{(1-\tau)\delta}{2\sigma_2^{n_1+(a+b)r+a}}$,
  \item for any $z\in \Pi(c)$, $B^u(z,L)\cap \Pi(c)$ has only one connected component which contains $z$.
\end{enumerate}

Now we describe a strategy for Alice to win the $(a, b)$-modified Schmidt games on $W^u(x)$ with target set $S=E_x(f, y)$. We claim that for each $j\in \mathbb{N}$, Alice can ensure that for any $x\in \psi(\omega'_{r(j+1)})$ and any $I_k=I_k(c)$ with $k< (j+1)r(a+b)$, she has $x \notin I_k$. This will imply $\cap_i\psi(\omega'_i) \subset (\bigcup_{k} I_k)^c \subset E_x(f, y)$ and finish the proof.

Let us prove the claim by induction on $j$. Note that each step consists of $r$ turns of play. Consider $j=0$. Suppose that Bob has chosen an atom $\psi(z_1,n_1)$ as above. So Alice needs to avoid all $I_k$'s with $0\leq k< r(a+b)$ in the next $r$ turns of play. For a given $0\leq k< r(a+b)$, we have
\begin{equation*}
\begin{aligned}
\text{diam}(f^k(\psi(z_1,n_1))) &\leq \text{diam}(f^{-(n_1-k)}(f^{n_1}(\psi(z_1,n_1))))\\
&\leq \frac{2(1+\tau)\delta}{\sigma_1^{n_1-k}} <\frac{2(1+\tau)\delta}{\sigma_1^{n_1-(a+b)r}}\leq \frac{L}{100}
\end{aligned}
\end{equation*}
by \eqref{e:n1}. This implies that $f^k(\psi(z_1,n_1)) \cap \Pi(c)$ has at most one connected component by the choice of $c$ and $L$. In other words, for each $0\leq k< r(a+b)$ there is at most one $I_k$ intersecting with $\psi(z_1,n_1)$. So there are at most $r(a+b)$ many $I_k$'s with different $k$'s intersecting with $\psi(z_1,n_1)$. For each of them one has
$$\text{diam}(I_k) \leq \frac{c}{\sigma_1^k} \leq c \leq \frac{(1-\tau)\delta}{2\sigma_2^{n_1+(a+b)r+a}}$$
by the choice of $c$. This guarantees that Alice can apply Lemma \ref{avoid} in each of the next $r$ turns of play to avoid some of these $I_k$'s. Since $r(a+b)(1-\eta)^r<1$, after $r$ turns of play Alice can avoid all these $I_k$'s. So the claim is true when $j=0$.

Assume the claim is true for $0, 1, \cdots, j-1$. Now we consider the $j$th step. Suppose that Bob already picked $\psi(\omega_{jr+1})$. In this step Alice only needs to avoid the $I_k$'s satisfying
\begin{equation}\label{e:ks}
jr(a+b) \leq k < (j+1)r(a+b)
\end{equation}
and
\begin{equation} \label{e:intersect}
I_k \cap \psi(\omega_{jr+1})  \neq \emptyset.
\end{equation}
For a given $k$ with \eqref{e:ks},
\begin{equation*}
\begin{aligned}
\text{diam}(f^k(\psi(\omega_{jr+1}))) &= \text{diam}(f^k(D_{n_1+jr(a+b)}^i))\\
&\leq \frac{2(1+\tau)\delta}{\sigma_1^{n_1+jr(a+b)-k}} \\
&\leq\frac{2(1+\tau)\delta}{\sigma_1^{n_1-(a+b)r}}\\
&\leq \frac{L}{100}
\end{aligned}
\end{equation*}
by \eqref{e:n1} and \eqref{e:ks}. This implies that $f^k(\psi(\omega_{jr+1})) \cap \Pi(c)$ has at most one connected component by the choice of $c$ and $L$. So there are at most $r(a+b)$ many $I_k$'s satisfying \eqref{e:ks} and \eqref{e:intersect}. We have at most  $r(a+b)$ many $f^{jr(a+b)}(I_k)$'s intersecting $f^{jr(a+b)}(\psi(\omega_{jr+1}))$ by applying $f^{jr(a+b)}$. As $0\leq k-jr(a+b)\leq r(a+b)$ and $f^{jr(a+b)}(\psi(\omega_{jr+1}))$ is an $n_1$-atom, we can use the same argument as in the case $j=0$ for Alice to get a choice in the picture of $f^{jr(a+b)}(\psi(\omega_{jr+1}))$ first, and then applying $f^{-jr(a+b)}$ back to get a choice in the picture of $\psi(\omega_{jr+1})$ in each turn of play. After $r$ turns of play Alice can avoid all these $I_k$'s satisfying \eqref{e:ks} and \eqref{e:intersect}. So the claim is true and the theorem follows.
\end{proof}

\begin{proof}[Proof of Corollary \ref{corollary}]
By Theorem \ref{winning} and (1) of Proposition \ref{winproperty}, the winning set $E_{x}(f, y)$ is dense in $W^u(x)$. By Proposition \ref{schmidtgame}, one has $\dim_H(E_{x}(f, y) \cap U) = u$.
\end{proof}

\begin{proof}[Proof of Corollary \ref{corollary2}]
By Theorem \ref{winning} and (2) of Proposition \ref{winproperty}, $E_x(f, Y)$ is an $a$-winning set. Then the proof of Corollary \ref{corollary} also applies to $E_x(f, Y)$.
\end{proof}

\section{Full Hausdorff dimension on the manifold}
\subsection{Measurability of $\{\mu_x\}$}
 In this section we prove Main Theorem \ref{main2}. We note that we can not apply Marstrand Slicing Theorem (cf. \cite{KM}) directly to get the estimate $\dim_H(E(f,y)) \geq \dim_H(W) + \min_{x\in W}\dim_H(E_x(f,y))$, as the unstable foliation is only absolutely continuous but not Lipschitz continuous (cf. \cite{BP}, \cite{PS}, \cite{BW}). We follow the technique in \cite{Wu} to construct a Borel measure $\mu$ on $E(f,y)$ and apply the following easy half of Frostman's Lemma (cf. \cite[p.60]{F}).

\begin{lemma} \label{frostman}
Let $F$ be a Borel subset of a Riemannian manifold $X$. Let $\mu$ be a Borel measure on $F$, and let $0 <c< \infty$ be a constant. If
\begin{equation*}
\limsup_{r \to 0} \frac{\mu(B(x, r))}{r^h} < c \ \text{\ for all\ } x\in F,
\end{equation*}
then $\dim_H F \geq h$.
\end{lemma}

Recall that $E^u\subset TM$ is the $u$-dimensional unstable distribution. Suppose that $D_0 \subset W^u(x)$ has diameter with respect to $d^u$ (and hence the diameter with respect to $d$) less than the injectivity radius of $M$. Let $E^{\bot}$ be the $(n-u)$-dimensional orthogonal complement of $E^u$. For any $z\in D_0$ define $W_\delta(z):=\exp_z(E^{\bot}_\delta(z))$, where $E^{\bot}_\delta(z)$ is the $\delta$-ball centered at origin in the subspace $E^{\bot}(z)$. It is clear that the restriction of $E^u$ to $D_0 \subset W^u(x)$ is smooth and hence $E^{\bot}(z)$ is smooth on $D_0$. Suppose that $\delta$ is less than the injectivity radius of $M$, then $\{W_\delta(z): z\in D_0\}$ form a local smooth foliation. Such a local smooth foliation can be constructed in a small neighborhood of any point in $M$ and it is transverse to $W^u$.

Fix $x_0 \in M$. Suppose $W$ is a local smooth foliation transverse to $W^u$ near $x_0$ which can be constructed as above. Recall that a $\delta$-rectangle at $x_0$ is defined as
\begin{equation*}
\Pi(x_0) := \Pi(x_0, W, \delta):= \bigcup_{x\in \overline{B^W(x_0,\delta/2)}} B^u(x, \delta/2).
\end{equation*}
Recall that we have much freedom to choose the local tilings $\mathcal{T}(x)$ on each local unstable manifold $B^u(x,\delta/2)$. We can specify a choice of $\mathcal{T}(x)$ such that the atoms vary measurably as $x$ varies in the following sense:
\begin{proposition}\label{tilingchoice}
There exist a sequence of finite partitions $\mathcal{Q}^{(n)}$ of $\overline{B^W(x_0, \delta/2)}$ and a family of $f$-induced $(D_0(x), \delta, \tau)$-local tilings on $W^u(x)$ (denoted as $\mathcal{T}(x)$) for each $x\in \overline{B^W(x_0,\delta/2)}$ such that:
\begin{enumerate}
  \item for each element $Q_i^{(n)} \in \mathcal{Q}^{(n)}$, the interior of $Q_i^{(n)}$, denoted as $\text{Int}(Q_i^{(n)})$, is an open and connected set, and $\nu(\partial Q_i^{(n)})=0$;
  \item $\mathcal{Q}^{(n)} \leq \mathcal{Q}^{(n+1)}$;
  \item for each $Q_i^{(n)}\in \mathcal{Q}^{(n)}$, there exists $x_i^n\in \text{Int}(Q_i^{(n)})$ such that for any $x\in \text{Int}(Q_i^{(n)})$ one has
  $$f^n(\overline{D_n^i(x)})=h_{x_i^n,x}(f^n(\overline{D_n^i(x_i^n)}))$$
  where $D_n^i(x)$ is an arbitrary $n$-atom in $\mathcal{T}(x)$ and $h_{x_i^n,x}$ is a holonomy map of some local smooth foliation.
\end{enumerate}
\end{proposition}
\begin{proof}
At first consider $n=0$. At $x_0$, we let $D_0(x_0):=B^u(x_0, \delta/2)$.  We have a local smooth foliation $W_\delta(z), z\in D_0(x)$ perpendicular to $D_0(x)$ as constructed above. For any $x\in \overline{B^W(x_0, \delta/2)}$, let $h^W_{x_0x}$ be the holonomy map of the smooth foliation from $B^u(x_0, \delta/2)$ to $B^u(x, \delta/2)$. Set $D_0(x):=h^W_{x_0x}(D_0(x_0))$. If $\delta$ is small enough, there exists $\tau>0$ such that $B^u(x, \frac{(1-\tau)\delta}{2}) \subset D_0(x) \subset B^u(x, (1+\tau)\delta)$ for any $x\in \overline{B^W(x_0, \delta/2)}$ since $h^W_{x_0x}$ is smooth. We set $\mathcal{Q}^{(0)}:=\{\overline{B^W(x_0, \delta/2)}\}$ and $\mathcal{T}_0(x):=D_0(x)$.

Suppose the conclusion is true for $0, \cdots, n-1$, and now we prove it for $n$. For each $Q_i^{(n-1)}\in \mathcal{Q}^{(n-1)}$ let $x$ be an arbitrary point in $\overline{Q_i^{(n-1)}}$. We choose a local tiling $\mathcal{T}_n(x)=\{D_n^i(x), i=1, \cdots, k_n(x)\}$ (for example, the one constructed by a maximal separated set as in Example \ref{exampletiling}). For each $D_n^i(x)$ consider $f^n(D_n^i(x))$ and construct a local smooth foliation as in case $n=0$ above. Thus, there exists a small enough $\delta_n^i(x)>0$ such that for any $x' \in \overline{B^W(x, \delta_n^i(x))}$ we have $f^n(\overline{D_n^i(x')})=h_{n,x,x'}^i(f^n(\overline{D_n^i(x)}))$ where $h_{n,x,x'}^i$ is the holonomy map of the local smooth foliation. Set $\delta_n(x):=\min_{1\leq i \leq k_n(x)}\delta_n^i(x)$. It is also worth noting that $h_{n,x,x'}^i$ of different $i$ coincide on the boundary of atoms for any $x'\in \overline{B^W(x, \delta_n(x))}$ because we constructed all local smooth foliations by using exponential maps and complement of $E^u$. Now we have an open cover $\{B^W(x, \delta_n(x)): x\in \overline{Q_i^{(n-1)}}\}$ of the compact space $\overline{Q_i^{(n-1)}}$. Thus we have a finite open cover from which we can obtain a finite partition $\mathcal{R}_i^{(n-1)}$ of $\overline{Q_i^{(n-1)}}$. We collect all the elements of $\mathcal{R}_i^{(n-1)}$ and obtain a finite partition of $\overline{B^W(x_0, \delta/2)}$. Set $\mathcal{Q}^{(n)}$ to be this partition and one can check (1) and (2) in the proposition easily. We have also constructed $\mathcal{T}_n(x)$ above for any $x\in \text{Int}(Q_i^{(n)})$ satisfying (3) in the proposition. Continue the induction and the proposition follows.
\end{proof}

The measure $\mu$ will be constructed to be supported on $\Pi(x_0)\cap E(f,y)$. At the first stage we construct the "conditional" measures on local unstable manifolds. Combining Proposition \ref{schmidtgame} and Theorem \ref{winning}, we can construct a family of measures $\{\mu_x\}$:
\begin{proposition} \label{familymeasure}
For any $\epsilon >0$, any $x\in \overline{B^W(x_0,\delta/2)}$, there exists a sequence of measures $\{\mu_x^{(l)}\}_{l=0}^{\infty}$ and $\{\mu_x\}$ such that:
\begin{enumerate}
  \item $\mu_x^{(l)}\text{\ is supported on\ } B^u(x, \delta/2), \text{\ and\ }\mu_x \text{\ is supported on\ } B^u(x, \delta/2) \cap E_x(f,y);$
  \item $\mu_x \text{\ is the unique weak limit of\ } \mu_x^{(l)};$
  \item for any $z\in B^u(x, \delta/2)\cap \text{supp}(\mu_x)$, $r>0$ small enough,
\begin{equation*}
\mu_x (B^u(z, r)) \leq Cr^{u-\epsilon}.
\end{equation*}
\end{enumerate}
\end{proposition}
\begin{remark}
Since the construction of $\mu_x^{(l)}$ and $\mu_x$ in Proposition \ref{familymeasure} relies on the construction of $\mathcal{A}_{x}^{(l)}$, they are not unique.
\end{remark}

Now we shall show that there exist choices of each $\mu_x^{(l)}$ and $\mu_x$ such that $x\mapsto \mu_x^{(l)}$ and $x\mapsto \mu_x$ are measurable with respect to the volume measure $\nu$ on $W$. So we need to specify the choice of the families $\mathcal{A}_x^{(l)}$ for any $l$ and any $x\in \overline{B^W(x_0, \delta/2)}$.

\begin{proposition}\label{measurablechoice}
There exist a sequence of finite partitions $\mathcal{P}^{(l)}$ of $\overline{B^W(x_0, \delta/2)}$, a family $\mathcal{A}_x=\{\mathcal{A}_{x}^{(l)}\}_{l=0}^{\infty}$ of collections of subsets of $B^u(x, \delta/2)$ and a family of measures $\{\mu_x^{(l)}\}$ such that:
\begin{enumerate}
  \item for each element $P_i^{(l)} \in \mathcal{P}^{(l)}$, the interior of $P_i^{(l)}$, denoted as $\text{Int}(P_i^{(l)})$, is an open and connected set, and $\nu(\partial P_i^{(l)})=0$;
  \item $\mathcal{P}^{(l)} \leq \mathcal{P}^{(l+1)}$;
  \item for each $x\in \overline{B^W(x_0, \delta/2)}$, each $l \in \mathbb{N}_0$, $\mu_x^{(l)}$ is supported on $\mathbf{A}_x^{(l)}$ and obtained as in Proposition \ref{familymeasure}. Moreover, $\bigcup_{x\in \text{Int}(P_i^{(l)})} \mathbf{A}_x^{(l)}$ is a union of finitely many open and connected sets.
  \item For any $l\in \mathbb{N}_0$, $x \mapsto \mu_x^{(l)}$ is continuous on each $\text{Int}(P_i^{(l)})$; hence it is measurable with respect to the volume measure $\nu$ on $\overline{B^W(x_0, \delta/2)}$.
\end{enumerate}
\end{proposition}
\begin{proof}
Suppose that we have a family of local tilings $\mathcal{T}(x)$ on each unstable manifold $B^u(x,\delta/2)$ which is constructed in Proposition \ref{tilingchoice}. We will make use of the finite partitions $\mathcal{Q}^{(n)}$ constructed there. Another useful family of finite partitions $\mathcal{S}^{(l)}$ is constructed below.

Recall the proof in Theorem \ref{winning}. Let us suppose $jr < l+1 \leq (j+1)r$, i.e., we are at $(l+1)$th turn of play and $j$th step in the proof of Theorem \ref{winning}. At $j$th step, Alice only needs to avoid some of the $I_k$'s with $jr(a+b) \leq k < (j+1)r(a+b)$. Let $\pi: \Pi(x_0)\rightarrow \overline{B^W(x_0, \delta/2)}$ be the natural map such that $\pi(z)$ is the unique point in $B^u(z,\delta/2) \cap \overline{B^W(x_0, \delta/2)}$. Consider the boundaries of the projection to $\overline{B^W(x_0, \delta/2)}$ under $\pi$ of all these $k$th preimages of $\Pi(c)$, i.e.
$$\bigcup_{k < (j+1)r(a+b)} \partial \left(\pi \left(f^{-k}(\Pi(c)) \cap \Pi(x_0)\right)\right).$$
Let $\mathcal{S}^{(l)}$ be the partition of $\overline{B^W(x_0, \delta/2)}$ into connected sets such that
$$\bigcup_{S\in \mathcal{S}^{(l)}}\partial S=\bigcup_{k < (j+1)r(a+b)} \partial \left(\pi \left(f^{-k}(\Pi(c)) \cap \Pi(x_0)\right)\right)\bigcup \partial(B^W(x_0, \delta/2)).$$
Then by definition of $\mathcal{S}^{(l)}$, for each $S\in \mathcal{S}^{(l)}$ and the $k$'s we need to consider, the set $f^{-k}(\Pi(c))\bigcap \bigcup_{x\in S}B^u(x,\delta'/2)$ is the union of finitely many open and connected sets (in general the intersection of two connected sets may have infinitely many connected components. To exclude this extreme scenario, we enlarge $\delta$ a little bit to $\delta'$ such that $I_k(x) \subset B^u(x,\delta'/2)$. Then $f^{-k}(\Pi(c))\bigcap \bigcup_{x\in S}B^u(x,\delta'/2)$ should have finitely many connected components). Since the number of $k$'s is finite, the partition $\mathcal{S}^{(l)}$ is finite and $\mathcal{S}^{(l-1)} \leq \mathcal{S}^{(l)}$. Note that $\mathcal{S}^{(l)}$ can be trivial if $\pi \left(f^{-k}(\Pi(c))\cap \Pi(x_0)\right) \supset \overline{B^W(x_0, \delta/2)}$ for all $k < (j+1)r(a+b)$ that we need consider. Note also that $\mathcal{S}^{(l)}$ of different $jr < l+1 \leq (j+1)r$ are the same.

Now we are ready to prove the proposition by induction on $l$. At each turn we specify the choice for $\mathcal{A}_x^{(l)}$ according to the constructions in the proof of Proposition \ref{schmidtgame} and Theorem \ref{winning}.

Consider $l=0$. Set $\mathcal{P}^{(0)}=\mathcal{Q}^{(n_1)} \vee \mathcal{S}^{(0)}$. For each $\overline{P_i^{(0)}}\in \mathcal{P}^{(0)}$ and any $x\in \overline{P_i^{(0)}}$ we can let Bob choose $\psi_{x}(\omega_1) \subset D_0(x)$ in the way that $$f^n(\overline{\psi_{x}(\omega_1)})=h_{x_i^{n_1},x}(f^n(\overline{\psi_{x_i^{n_1}}(\omega_1)}).$$
by Proposition \ref{tilingchoice}. This implies that $\bigcup_{x\in \text{Int}(P_i^{(0)})}\psi_x(\omega_1)$ is open and connected by construction of $\mathcal{Q}^{(n_1)}$. Recall the proof in Theorem \ref{winning}. When $l=0$, Alice only needs to avoid some of the finitely many $I_k$'s with $0 \leq k < r(a+b)$ and recall that for each of such $k$ there exists only one $I_k$ to avoid. By definition of $\mathcal{S}^{(0)}$, $\bigcup_{x\in \text{Int}(P_i^{(0)})}\left(I_x^{(k)}\cap \psi_x(\omega_{1})\right)$ is open and connected for each of such $k$. By avoiding the above open and connected sets, Alice can choose atoms $\psi_x(\omega'_{1})$ such that $\bigcup_{x\in \text{Int}(P_i^{(0)})}\psi_x(\omega'_{1})$ is open and connected. Alice can do so because of the smoothness of the holonomy map in (3) of Proposition \ref{tilingchoice}. We proved (3) when $l=0$. (4) is immediate since $\mu_x^{(0)}$ is the volume measure on $\psi_x(\omega'_1)$ according to the proof of Lemma \ref{treelike} (see \cite{Mc}, \cite{U}, \cite{Wu}).

Suppose the conclusion is true for $0, \cdots, l-1$, and now we prove it for $l$. Set $\mathcal{P}^{(l)}=\mathcal{P}^{(l-1)} \vee \mathcal{Q}^{(n_{l+1)}}\vee \mathcal{S}^{(l)}$. We proved (1) and (2). The remaining proof of (3) and (4) is essentially just a combination of induction hypothesis and a repetition of the argument for $l=0$. Since $\mathcal{A}_x^{(l)}$ will be constructed in the $(l+1)$th turn of play, let us suppose $jr < l+1 \leq (j+1)r$, i.e., we are at the $j$th step in the proof of Theorem \ref{winning}. At $j$th step, Alice only needs to avoid some of the $I_k$'s with $jr < l+1 \leq (j+1)r$. Consider $P_i^{(l)} \in \mathcal{P}^{(l)}$. Since $\mathcal{P}^{(l-1)} \leq \mathcal{P}^{(l)}$, by the induction hypothesis $\bigcup_{x\in \text{Int}(P_i^{(l)})} \mathbf{A}_x^{(l-1)}$ is a finite union of open and connected sets. So inside each of Alice's balls in the $l$th turn, we can let Bob choose $N$ atoms $\psi_x(\omega_{l+1})$ on the $(l+1)$th turn such that each $\bigcup_{x \in \text{Int}(P_i^{(l)})}\psi_x(\omega_{l+1})$ is open and connected by construction of $\mathcal{Q}^{(n_{l+1})}$. Then by the definition of $\mathcal{S}^{(l)}$, $\bigcup_{x\in \text{Int}(P_i^{(l)})}\left(I_x^{(k)}\cap \psi_x(\omega_{l+1})\right)$ is open and connected for each of such $I_k$ that Alice needs to avoid. By avoiding the above open and connected sets, Alice can choose balls $\psi_x(\omega'_{l+1})$ such that $\bigcup_{x\in \text{Int}(P_i^{(l)})}\psi_x(\omega'_{l+1})$ is open and connected.  Alice can do so because of the smoothness of the holonomy map in (3) of Proposition \ref{tilingchoice}. We proved (3). (4) is immediate from (3) and the fact that $\mu_x^{(l)}$ is obtained by a rescaling of the volume measure in each of the $l$ steps according to the proof of Lemma \ref{treelike} (see \cite{Mc}, \cite{U}, \cite{Wu}).
\end{proof}

\begin{proposition}\label{measurable}
Fix an arbitrary small $\epsilon >0$. For each $x\in \overline{B^W(x_0, \delta/2)}$, there exists a measure $\mu_x$ supported on $B^u(x, \delta/2) \cap E_x(f,y)$ such that
\begin{enumerate}
\item For any $z\in B^u(x, \delta/2)$, $\mu_x (B^u(z, r)) \leq Cr^{u-\epsilon}$,
\item $x\mapsto \mu_x$ is measurable with respect to the volume measure $\nu$ on $\overline{B^W(x_0, \delta/2)}.$
  \end{enumerate}
\end{proposition}

\begin{proof}
Let $\mu_x^{(l)}$ be as in Proposition \ref{measurablechoice}. By Proposition \ref{familymeasure}, we can take $\mu_x$ as the unique weak limit of $\mu_x^{(l)}$. (1) is immediate from (3) of Proposition \ref{familymeasure}. Since $x \mapsto\mu_x^{(l)}$ is measurable by (4) of Proposition \ref{measurablechoice}, $x\mapsto\mu_x$ is measurable and we get (2).
\end{proof}

\subsection{Measure $\mu$}
Finally we can define the measure $\mu$ supported on $\Pi(x_0)\cap E(f,y)$.
\begin{definition} \label{measure}
For any Borel set $A \subset M$, define
\begin{equation*}
\mu(A):= \int_{\overline{B^W(x_0, \delta/2)}}\int_{B^u(x, \delta/2)} \chi_A(z)d\mu_x(z) d\nu(x).
\end{equation*}
where $\chi_A$ is the characteristic function of $A$, $\nu$ is the volume measure on $\overline{B^W(x_0, \delta/2)}$, and $\mu_x$ is as in Proposition \ref{measurable}. $\mu$ is well defined since $x \mapsto \mu_x$ is measurable with respect to $\nu$.
\end{definition}

Fix $0 <\tau <1$. For any $z \in \Pi(x_0)$, define
\begin{equation*}
B_l(z):=B(z, \tau^l).
\end{equation*}
and
\begin{equation*}
C_l(z):=\bigcup_{u \in B^W (z, \tau^l)}B^u(u, \tau^l).
\end{equation*}
From the definition of $\mu$, it is easier to estimate $\mu$-measure of sets $C_l(z)$ than balls $B_l(z)$. Nevertheless the following lemma allows us to estimate $\mu$-measure of balls. The same lemma has appeared in \cite{Wu} (Lemma 5.10), so the proof is omitted here.
\begin{lemma}
There exists a $l_0 >0$ such that for any $l \geq l_0$ and any $z\in \Pi(x_0)$ we have
\begin{equation*}
 C_{l}(z) \subset B_{l-l_0}(z) \text{\ \ and\ \ } B_{l}(z) \subset C_{l-l_0}(z).
\end{equation*}
\end{lemma}

Now we are able to estimate $\mu$-measure of balls in the following lemma. The proof is exactly same as the one of Lemma 5.11 in \cite{Wu} and we omitted here.
\begin{lemma}\label{frostman2}
For any $r>0$ small enough and any $z\in \Pi(x_0)$, one has
\begin{equation*}
\mu(B(z,r)) \leq D r^{n-\epsilon}
\end{equation*}
for some constant $0<D<\infty$.
\end{lemma}

\begin{proof}[Proof of Main Theorem \ref{main2}]
For any nonempty open subset $V \subset M$, we can find some $x_0 \in V$, a local smooth foliation $W$ transverse to $W^u$ near $x_0$, and $\delta>0$ small enough such that $\Pi(x_0, W, \delta) \subset V$. Main Theorem \ref{main2} follows immediately from Lemma \ref{frostman}, Lemma \ref{frostman2} and by letting $\epsilon \to 0$.
\end{proof}

Proof of Corollary \ref{countable} is a modification of the one of Corollary 2.4 in \cite{Wu}. For completeness, we provide the proof here:
\begin{proof}[Proof of Corollary \ref{countable}]
Let $Y=\{y_t\}_{t=1}^\infty$. Then $E_x(f, Y)=\cap_{t=1}^\infty E_x(f, y_t)$ is also a winning set for modified Schmidt games induced by $f$ on $W^u(x)$ with respect to $\mathcal{T}$. Let's recall a proof of this fact due to \cite{S} which we omitted in the proof of Proposition \ref{winproperty}. Let $a$ be as in \eqref{e:a}. We know that $E_x(f, y_t)$ is $a$-winning for all $t \in \mathbb{N}$, and we want to show that $E_x(f, Y)$ is $(a, b)$-winning for any $b>a_*$. Here is the strategy for Alice to win. At the first, third, fifth, \ldots turns, Alice uses an $(a, a+2b;E_x(f, y_1))$-winning strategy which forces $\cap_{j=1}^\infty \psi(\omega_{1+2(j-1)}) \subset E_x(f, y_1)$. At the second, sixth, tenth, \ldots turns, Alice uses an $(a, b+3(a+b);E_x(f, y_2))$-winning strategy which forces $\cap_{j=1}^\infty \psi(\omega_{2+2^2(j-1)}) \subset E_x(f, y_2)$. In general, at $k$th turn with $k\equiv 2^{t-1} (\text{mod}2^t)$, Alice uses an $(a, b+(2^t-1)(a+b);E_x(f, y_t))$-winning strategy which forces $\cap_{j=1}^\infty \psi(\omega_{2^{t-1}+2^t(j-1)}) \subset E_x(f, y_t)$. By this strategy, Alice can enforce that the unique point in the intersection of all atoms is in $E_x(f, Y)$.

Now we construct the measure $\mu$ supported on $\Pi(x_0)\cap E(f,Y)$ as in Definition \ref{measure}. We can construct $\mu_x^{(l)}$ and $\mu_x$ as in Proposition \ref{familymeasure} supported on $B^u(x, \delta/2) \cap E_x(f, Y)$ since $E_x(f, Y)$ is a winning set for modified Schmidt games. So it is enough to specify a choice such that $x\mapsto \mu_x$ is measurable. The idea is same as in the proof of Proposition \ref{measurablechoice}. The difference is that at $l$th step of the induction(i.e. $(l+1)$th turn of the game) with $l+1\equiv 2^{t-1} (\text{mod}2^t)$, Alice needs to avoid some $I_x^{(k)}(y_t)$'s, the preimages of $\Pi(y_t, c_t)$, where $\Pi(y_t, c_t)$ is the open rectangle neighborhood of $y_t$ as in the proof of Theorem \ref{winning} for the $(a, b+(2^t-1)(a+b);E_x(f, y_t))$-modified Schmidt games. If $l+1=2^{t-1}+2^t(j-1)$, then $k$ is bounded above by some finite number and thus there are only finitely many such $k$'s. So the argument in Proposition \ref{measurablechoice} still works, i.e., at $(l+1)$th turn, there exists a finite partition $\mathcal{P}^{(l)}$, and Alice can choose atoms such that $\bigcup_{x\in P_i^{(l)}} \mathbf{A}_x^{(l)}$ is a finite union of open and connected sets. Hence we have that $x\mapsto \mu_x$ is measurable as before.
\end{proof}
\begin{remark}
Let $\{f_i\}_{i=1}^N$ ($N\geq 2$) be any finite set of $C^{1+\theta}$-partially hyperbolic diffeomorphsims on $M$, and $y \in M$. Since our definition of modified Schmidt games depends on the partially hyperbolic diffeomorphism itself, it is nonsense to say whether $\cap_{i=1}^N E_x(f_i,y)$ is winning for modified Schmidt games. Thus it is not known whether $\cap_{i=1}^N E(f_i,y)$ has full Hausdorff dimension on $M$ when $N \geq 2$.
\end{remark}

\ \
\\[-2mm]
\textbf{Acknowledgement.} This is a subsequent work of my PhD thesis and I would like to thank my advisors Federico Rodriguez Hertz and Anatole Katok for their constant support.

\end{document}